\documentclass[reqno]{amsart}

\usepackage{amsmath,amssymb,amsthm,mathtools}
\usepackage[margin=1in]{geometry}
\usepackage{enumitem}
\usepackage{microtype}
\usepackage{xcolor}
\usepackage{hyperref}
\usepackage[capitalize,nameinlink]{cleveref}
\hypersetup{
  colorlinks=true,
  linkcolor=blue,
  citecolor=red,
  urlcolor=blue
}

\newtheorem{theorem}{Theorem}[section]
\newtheorem{proposition}[theorem]{Proposition}
\newtheorem{lemma}[theorem]{Lemma}
\newtheorem{claim}[theorem]{Claim}

\theoremstyle{definition}
\newtheorem{definition}[theorem]{Definition}
\theoremstyle{remark}

\crefname{proposition}{Proposition}{Propositions}
\crefname{claim}{Claim}{Claims}

\newcommand{\F}{\mathbb F}
\newcommand{\R}{\mathbb R}
\newcommand{\E}{\mathbb E}
\renewcommand{\P}{\mathbb P}
\newcommand{\cQ}{\mathcal Q}
\newcommand{\cH}{\mathcal H}
\newcommand{\one}{\mathbf 1}
\newcommand{\qbinom}[2]{\genfrac{[}{]}{0pt}{}{#1}{#2}_{\!2}}
\DeclareMathOperator{\Aff}{Aff}
\DeclareMathOperator{\Gr}{Gr}
\DeclareMathOperator{\Span}{span}
\DeclareMathOperator{\supp}{supp}
\DeclareMathOperator{\dir}{dir}
\DeclareMathOperator{\codim}{codim}
\DeclareMathOperator{\rank}{rank}

\usepackage[
backend=biber,
style=numeric,
maxnames=99,
giveninits=true
]{biblatex}
\renewbibmacro{in:}{%
  \ifentrytype{article}{}{\printtext{\bibstring{in}\intitlepunct}}}

\title{Asymptotically sharp bounds for affine subspace statistics in $\F_2^n$}

\makeatletter
\newcommand\thankssymb[1]{\textsuperscript{\@fnsymbol{#1}}}
\makeatother

\author[T.-W. Chao]{Ting-Wei Chao\thankssymb{1}}
\author[Z. Xu]{Zixuan Xu\thankssymb{1}}
\author[D. Zakharov]{Dmitrii Zakharov\thankssymb{1}}

\thanks{\thankssymb{1}Department of Mathematics, Massachusetts Institute of Technology,
Cambridge, MA, USA. Emails: \texttt{\{twchao, zixuanxu, zakhdm\}@mit.edu}.}

\begin{document}

\begin{abstract}
Given a subset $A \subseteq \F_2^n$, we can consider the distribution of the intersection size of $A$ with a uniformly random $d$-flat $F$. Motivated by the edge statistics problem and the hypercube statistics problem, the affine subspace statistics problem concerns the maximum of $\mathbb{P}[|F\cap A|=s]$ among $A \subseteq \F_2^n$ for any fixed $s\in\{1,\dots,2^d\}$ over a uniformly random $d$-flat $F$. We use $\lambda^*(d,s)$ to denote the limit of the maximum when $n$ goes to infinity.

In this note, we prove tight bounds for $\lambda^*(d,s)$ in two different regimes.
For $s=j2^k$ where $j$ is a positive odd integer, the best known lower bound construction achieving $\lambda^*(d,s)\ge 1-2^{-k}$ is due to taking $A$ as the union of $j$ parallel $(n-d+k)$-flats in $\F_2^n$. Our main result is a matching upper bound with an additive error term of $O(2^{-3k/2})$.
We also study the case $s=1$, where we determine $\lambda^*(d,1)$ exactly. We show that the random construction where each point is included with probability $2^{-d}$ is optimal.
\end{abstract}

\maketitle

\section{Introduction}

Let $A\subseteq \F_2^n$ be a subset. For integers $d\geq 0$ and $1\leq s\leq 2^d$, let $\lambda^*(n,d,s,A)$ be the fraction of $d$-flats ($d$-dimensional affine subspaces) that intersect $A$ in exactly $s$ points. What is the maximum possible value of $\lambda^*(n,d,s,A)$ over all possible choices of $A$?

This question is motivated by the {\em hypercube statistics} problem initiated by Alon, Axenovich, and
Goldwasser \cite{AlonAxenovichGoldwasser2024}, where instead of all $d$-flats, one considers only coordinate axis-aligned $d$-flats (i.e. $d$-dimensional subcubes). They obtained general upper and lower bounds for various values of $d$ and $s$ and conjectured that for $s=1$ the maximum fraction tends to $1/e+o_d(1)$ as $d,n \to \infty$ . Similar inducibility problems for graphs and hypergraphs are well-studied in works on edge statistics including \cite{AlonHefetzKrivelevichTyomkyn2020,KwanSudakovTran2019,MartinssonMoussetNoeverTrujic2019,FoxSauermann2020,JainKwanMubayiTran2025,GrebennikovKwan2025}.
In particular, the edge statistics conjecture posed in Alon--Hefetz--Krivelevich--Tyomkyn~\cite{AlonHefetzKrivelevichTyomkyn2020}, which was fully resolved by the works of Kwan--Sudakov--Tran~\cite{KwanSudakovTran2019} and Fox--Sauermann~\cite{FoxSauermann2020}, states that for all integers $k$ and $\ell \in \{1, \dots, {k\choose 2}\}$, the probability of a uniformly random subset of $k$ vertices induces exactly $\ell$ edges in any graph is at most $1/e+o_k(1)$.

In this paper, we study the basis-invariant version of the hypercube statistics problem, introduced in \cite{Xu2026}. We begin by formally defining the problem. Let $n\ge d\ge1$, $0\le s\le2^d$, and
let $A\subseteq\F_2^n$. Define
\[\lambda^*(n,d,s,A):=\P_{F\in\Aff(n,d)}\bigl[|F\cap A|=s\bigr],\]
where $\Aff(n,d)$ is the set of all $d$-flats in $\F_2^{n}$ and $F$ is chosen uniformly at random from $\Aff(n,d)$. We then set
\[\lambda^*(n,d,s):=\max_{A\subseteq\F_2^n}\lambda^*(n,d,s,A)\quad\text{and}\quad \lambda^*(d,s):=\lim_{n\to\infty}\lambda^*(n,d,s).\]
The limit exists because $\lambda^*(n,d,s)$ is non-increasing in $n$. Indeed, fix a subset $A\subseteq\F_2^{n+1}$ and choose a uniformly random $n$-flat $H\subseteq\F_2^{n+1}$. Conditioning on $H$, choose a uniformly
random $d$-flat $F\subseteq H$. The marginal distribution of $F$ is uniform over $\Aff(n+1,d)$, since every $d$-flat is contained in the same
number of $n$-flats. Therefore,
\[\lambda^*(n+1,d,s,A)=\E_H\,\P_{F\subseteq H}[|F\cap(A\cap H)|=s]\le\lambda^*(n,d,s).\] 

We first recall the lower bound construction given in \cite{AlonAxenovichGoldwasser2024}, which was constructed for the hypercube statistic problem, but the construction gives the same bound for the affine subspace statistics problem. 

\begin{proposition}\label{prop:lower}
Let $0\le k\le d$ and let $s=j2^k\le2^d$, where $j$ is a positive odd
integer.  Then
\begin{equation}\label{eq:lambda}
  \lambda^*(d,s)\ge  1-\bigl(1-2^{-(d-k)}\bigr)2^{-k}+O(2^{-2k}).  
\end{equation}
\end{proposition}

The construction is to take $A$ as the union of $j$ parallel flats of codimension $d-k$. One can verify that $\lambda^*(d,s,A)\ge \prod_{i = k+1}^d(1-2^{-i})$ which implies (\ref{eq:lambda}).

For $s=j2^k$, \cite{Xu2026}
showed that $1-\lambda^*(d,j2^k)=\Theta(2^{-k})$. To be more precise, for fixed $k$ and $d\to\infty$, the upper bound was
\[
  \lambda^*(d,j2^k)\le 1-\frac23\bigl(1-2^{-(d-k)}\bigr)2^{-k}
  +O(2^{-2k})+o_d(1),
\]
so the leading coefficient in the second term is between
$2/3$ and $1$. Our main result is to prove an upper bound with the correct leading coefficient on the second term in the regime where $k$ is sufficiently large.

% It wasconjectured there that, for fixed $k$ and $d\to\infty$, the parallel-flatconstruction is asymptotically optimal. 

\begin{theorem}\label{thm:main}
Let $s=j2^k\le2^d$, where $j$ is a positive odd integer and $k$ is
sufficiently large.
Then
\[\lambda^*(d,s)\le 1-\bigl(1-2^{-(d-k)}\bigr)2^{-k}+O(2^{-3k/2}).\]
Consequently, together with \cref{prop:lower},
\[\lambda^*(d,s) =1-\bigl(1-2^{-(d-k)}\bigr)2^{-k}+O(2^{-3k/2}).\]
The implicit constants are absolute.
\end{theorem}

In particular, if $k\to\infty$ and $d-k\to\infty$, then
\[
  1-\lambda^*(d,j2^k)=(1+o(1))2^{-k}.
\]

Our second result determines the exact value of $\lambda^*(d,1)$. Note that taking $A$ to be a uniformly random set of density $2^{-d}$ in $\F_2^n$ gives the lower
bound $\lambda^*(d,1)\ge (1-2^{-d})^{2^d-1}$. The second author conjectured in \cite{Xu2026} that this construction is asymptotically optimal. We show that this is indeed optimal for any $d\geq 1$.

\begin{theorem}\label{thm:singleton}
For every integer $d\ge1$,
\[
  \lambda^*(d,1)=(1-2^{-d})^{2^d-1}.
\]
Consequently, $\lambda^*(d,1)\to e^{-1}$ as $d\to\infty$.
\end{theorem}
 
We emphasize that \cref{thm:singleton} concerns
all affine flats. For the analogous hypercube statistics problem for $s = 1$, it was conjectured by Alon, Axenovich, and Goldwasser in \cite{AlonAxenovichGoldwasser2024} that the random construction is asymptotically optimal. This problem is still open, but one may view \cref{thm:singleton} as a weaker result than proving the same bound for hypercube statistics.

\subsection{Paper Organization}
In \cref{sec:prelim}, we include definitions and facts used in our proof.  We present the proof of \cref{thm:main} in \cref{sec:main-proof} and the proof of \cref{thm:singleton} in \cref{sec:singleton-proof}.

\subsection{Acknowledgments.} We thank Dor Minzer for helpful discussions. D.Z. was supported by the Simons Dissertation fellowship.

AI tools have been used in the preparation of this manuscript. The proof of \cref{thm:singleton} is due to ChatGPT 5.5-Pro and was found when trying to generalize an argument for $d=2$ observed by the authors. The ideas in the proof of \cref{thm:main} are entirely human generated. ChatGPT 5.6 was used for editing and proofreading an initial draft of the current manuscript.

\section{Preliminaries}\label{sec:prelim}

Let $\Gr(n,d)$ be the set of $d$-dimensional linear subspaces of
$\F_2^n$, and let $\Aff(n,d)$ be the set of affine $d$-flats. Every
$F\in\Aff(n,d)$ has the form $F=x+U$ with $U\in\Gr(n,d)$; we write
$\dir(F)=U$. Recall the Gaussian binomial coefficient
\[
  \qbinom{n}{d}
  :=\prod_{i=0}^{d-1}\frac{2^{n-i}-1}{2^{d-i}-1}.
\]
It counts the number of elements in $\Gr(n,d)$, and consequently
\[
  |\Aff(n,d)|=2^{n-d}\qbinom{n}{d}.
\]
Note that a uniformly random affine $d$-flat can be sampled by first choosing
$U\in\Gr(n,d)$ uniformly and then choosing a uniformly random coset of
$U$.

\subsection{The lower-bound construction}

We now prove \cref{prop:lower}.
Fix a surjective linear map
\[
  \pi:\F_2^n\longrightarrow\F_2^{d-k}
\]
% and define
% \[
%   c_n(d,k):=
%   \P_{U\in\Gr(n,d)}[\rank(\pi|_U)=d-k].
% \]
and write $K=\ker \pi$, then we have $\dim K=n-d+k$. We first compute the probability $p$ that $\rank(\pi|_U)=d-k$, where $U$ is chosen uniformly at random from $\Gr(n,d)$. The restriction $\pi|_U$ is surjective if and only if $\dim(U\cap K)=k$. Thus, we may first choose
$U\cap K$, and then describe $U/(U\cap K)$ as the graph of a linear
map into $K/(U\cap K)$. This gives
\begin{equation}\label{eq:finite-rank-probability}
  p
  =\frac{2^{(d-k)(n-d)}\qbinom{n-d+k}{k}}{\qbinom{n}{d}}.
\end{equation}
% Expanding the Gaussian binomial coefficients shows that
% \begin{equation}\label{eq:rank-limit}
%   \lim_{n\to\infty}c_n(d,k)
%   =\prod_{i = k+1}^d(1-2^{-i})=:c(d,k).
% \end{equation}
Choose a set $T\subseteq\F_2^{d-k}$ of size $j$ and put
\[
  A:=\pi^{-1}(T).
\]
Thus, $A$ is a union of $j$ parallel affine flats, namely the fibers
$\pi^{-1}(t)$ for $t\in T$.  Let
$F=x+U$ be a $d$-flat. If $\pi|_U$ is surjective, then the affine map
$\pi|_F$ is also surjective and all of its fibers have size $2^k$.
Consequently,
\[
  |F\cap A|=|T|2^k=j2^k=s.
\]
It follows that
\[\lambda^*(n,d,s,A)\ge p
  =\frac{2^{(d-k)(n-d)}\qbinom{n-d+k}{k}}{\qbinom{n}{d}}.\]  
Letting $n\to\infty$ proves $\lambda^*(d,s)\ge \prod_{i = k+1}^d(1-2^{-i})$.

We have an elementary estimation
\[
\prod_i(1-x_i)=1-\sum_i x_i+O\biggl(\left(\sum_i x_i\right)^2\biggr),
\]
for $0\le x_i\le1/2$. It follows that $\prod_{i = k+1}^d(1-2^{-i}) = 1- (1-2^{-(d-k)})2^{-k} + O(2^{-2k})$.

\subsection{Fourier support and blocking sets}

Let $H$ be an $m$-dimensional vector space over $\F_2$, and let
$\widehat H=\operatorname{Hom}(H,\F_2)$ be its dual.  For
$f:H\to\R$, we use the unnormalized Fourier transform
\[
  \widehat f(\xi):=\sum_{x\in H}f(x)(-1)^{\xi(x)},
  \qquad \xi\in\widehat H.
\]
If $U\le H$ and $x\in H$, Fourier inversion gives the coset identity
\begin{equation}\label{eq:coset-fourier}
  \sum_{u\in U}f(x+u)
  =2^{-\codim_H U}
    \sum_{\xi\in U^\perp}\widehat f(\xi)(-1)^{\xi(x)},
\end{equation}
where
\[
  U^\perp:=\{\xi\in\widehat H:\xi(u)=0\text{ for every }u\in U\}.
\]

Over $\F_2$, the points in the projective space $\operatorname{PG}(m-1,2)$
may be identified with the nonzero vectors in an $m$-dimensional vector
space.  A projective $r$-flat is the set of nonzero vectors in an
$(r+1)$-dimensional linear subspace.  A set of projective points is a
\emph{blocking set with respect to $r$-flats} if it meets every
projective $r$-flat.  We use the following theorem of Bose and Burton.

\begin{theorem}[Bose--Burton \cite{BoseBurton1966}]\label{thm:bose-burton}
If $D\subseteq\operatorname{PG}(m-1,2)$ is a blocking set with respect
to projective $r$-flats, then $|D|\ge2^{m-r}-1$.
\end{theorem}

The next lemma gives a lower bound of the Fourier-support of the indicator function of a set that we will be using in the proof of \cref{thm:main}.  We include its short derivation from \cref{thm:bose-burton}. For $x \in \mathbb Z$, we denote $\nu_2(x)$ the 2-adic valuation of $x$, i.e. the largest integer $k$ such that $2^k\mid x$.

\begin{lemma}\label{lem:2adic-support}
Let $H$ be an $m$-dimensional vector space over $\F_2$ and $S\subseteq H$ be a nonempty subset. Let $r=\nu_2(|S|)$.  Then we have
\[
  |\supp(\widehat{\one_S})|\ge2^{m-r}.
\]
\end{lemma}

\begin{proof}
If $r=m$, then there is nothing to prove, so we can assume that $r < m$. 
Put $f=\one_S$ and let $L\le\widehat H$
have dimension $r+1$.  If
$\widehat f(\xi)=0$ for every $\xi\in L\setminus\{0\}$, then applying
\eqref{eq:coset-fourier} with $U=L^\perp$ shows that every coset of
$L^\perp$ contains exactly
\[
  2^{-(r+1)}\widehat f(0)=\frac{|S|}{2^{r+1}}
\]
points of $S$.  This is impossible because $2^{r+1}\nmid |S|$.
Therefore
\[
  \supp(\widehat f)\setminus\{0\}
\]
meets every $(r+1)$-dimensional subspace of $\widehat H$, and hence is a
blocking set with respect to projective $r$-flats in
$\operatorname{PG}(m-1,2)$.  By \cref{thm:bose-burton}, it has at least
$2^{m-r}-1$ elements.  Since $\widehat f(0)=|S|\ne0$, adding the zero
frequency proves the lemma.
\end{proof}

\section{The $s=j2^k$ case}\label{sec:main-proof}

We prove \cref{thm:main} using the following 
estimate for $n$ close to $d$. We will then deduce the theorem using the monotonicity of $\lambda^*(n,d,s)$.

\begin{proposition}\label{prop:finite}
Let $s=j2^k\le2^d$, where $j$ is a positive odd integer. Let
$2\le a\le k$ and set
$n=d+a$. Then
\[
  \lambda^*(n,d,s)
  \le 1-2^{-k}\bigl(1-2^{-(d-k)}\bigr)(1-2^{-a})
  \prod_{i=1}^{a-1}(1-2^{i-k}).
\]
\end{proposition}

\begin{proof}
Put $V=\F_2^n$ and fix $A\subseteq V$.  Call a $d$-flat $Q$
\emph{good} if $|Q\cap A|=s$ and \emph{bad} otherwise, and write
\[
  \rho:=\P_{Q\in\Aff(n,d)}[Q\text{ is bad}]
       =1-\lambda^*(n,d,s,A).
\]
If $d=k$, then the
right-hand side in the proposition is $1$, so we may assume that $d>k$.
If $\rho\ge2^{-k}$, then the conclusion is immediate.  We may therefore
assume that $\rho<2^{-k}\le1/4$.

For each $(d+1)$-flat $F$, define the local bad density as
\[
  \rho_F:=\P_{\substack{Q\subset F\\ \dim Q=d}}[Q\text{ is bad}].
\]
Note that we have
\[\rho=\E_{F\in\Aff(n,d+1)}\rho_F.\]
Hence there is a $(d+1)$-flat $W$ with $\rho_W\le\rho<2^{-k}$. For the rest of the proof we fix one such choice of $W$.
Replacing $A$ and $W$ by a common translate, we may further assume that $W$ is
a linear subspace.  Set $S=A\cap W$.

For every nonzero $\xi\in\widehat W$, the two level sets of $\xi$ are
disjoint parallel $d$-flats whose union is $W$.  If $|S|\ne2s$, at
least one member of every such pair is bad, which would give
$\rho_W\ge1/2$ contradicting our assumption that $\rho_W\le 1/4$.  Thus we must have $|S|=2s=j2^{k+1}$.

Let $f=\one_S$.  For a nonzero $\xi\in\widehat W$ and $b \in \F_2$, write
$H_{\xi,b}=\{x\in W:\xi(x)=b\}$.  Since $|S|=2s$, the two hyperplanes
$H_{\xi,0}$ and $H_{\xi,1}$ are either both good or both bad.  Moreover,
\[
  \widehat f(\xi)
  =|S\cap H_{\xi,0}|-|S\cap H_{\xi,1}|.
\]
Consequently, if $T=\supp(\widehat f)$, then
\[
  \rho_W=\frac{|T|-1}{2^{d+1}-1}.
\]
Since $\rho_W< 2^{-k}$, we have $|T|\le2^{d+1-k}$. On the other hand, applying \cref{lem:2adic-support} with $m=d+1$ and $\nu_2(|S|)=k+1$ gives $|T|\ge2^{d-k}$.

Thus, $W$ satisfies the following:
\begin{itemize}
    \item $|W\cap A| = j\cdot 2^{k+1}$,
    \item $2^{d-k}\le |\supp(\widehat{f})|\le 2^{d+1-k}$.
\end{itemize}

% We now count subspaces of $\widehat W$ which have a minimal nontrivial intersection with $T$.
Now we define the key object that we will be counting.

\begin{definition}
An $a$-dimensional subspace $L\le\widehat W$ is \emph{nice} if $L\cap T=\{0,\gamma\}$ for some nonzero $\gamma\in T$.
\end{definition}

For any $a$-dimensional subspace $L\le\widehat W$, let $\cQ_L$ be the set of affine $d$-flats $Q\subseteq V$ such that $Q\cap W$ is a coset of $L^\perp$. We need the following two counting claims about $\cQ_L$.

\begin{claim}\label{claim:extensions}
For any $a$-dimensional subspace $L\le\widehat W$, we have $|\cQ_L|=2^{a^2}$.
\end{claim}

\begin{proof}
For simplicity, we write $U=L^\perp$. Recall that we have $U\subseteq W\subseteq V$,  $\dim(W/U)=a$, and $\dim(V/W)=n-(d+1)=a-1$. For any $Q\in\cQ_L$, set $D=\dir(Q)$. By the definition of $\cQ_L$, we have $\dim(D)=d$ and $D\cap W=U$. Therefore, we know that $D/U$ and $W/U$ are linearly independent subspaces of $V/U$. We also have $\codim_{V/U}(D/U)=\codim_V(D)=a=\dim(W/U)$. Therefore, the number of feasible $D$ is the same as the number of linear maps from $(V/U)/(W/U)\simeq V/W$ to $W/U$, and there are $2^{a(a-1)}$ such maps. For every feasible $D$,
we have $D+W=V$, so each of the $2^{n-d}=2^a$ cosets of $D$ meets $W$
in a coset of $U$. Therefore, the total number is $2^{a(a-1)}2^a=2^{a^2}$.
\end{proof}

\begin{claim}\label{claim:forced-bad}
If $L$ is nice, then at least
$2^{a^2-a+1}$ of the flats in $\cQ_L$ are bad.
\end{claim}

\begin{proof}
Again, write $U=L^\perp$. Let $L\cap T=\{0,\gamma\}$ and choose $w\in W$ with
$\gamma(w)=1$.  By \eqref{eq:coset-fourier}, for every $x\in W$,
\[
  |A\cap(x+U)|
  =2^{-a}\left(\widehat f(0)
      +\widehat f(\gamma)(-1)^{\gamma(x)}\right).
\]
Since $\widehat f(\gamma)\ne0$, it follows that
\begin{equation}\label{eq:unequal-cosets}
  |A\cap(x+U)|\ne|A\cap(x+w+U)|.
\end{equation}

For $Q\in\cQ_L$, the vector $w$ does not lie in $\dir(Q)$, because
$\dir(Q)\cap W=U$ and $w\notin U$.  Hence
\[Q^+:=Q\cup(Q+w)\]
is a $(d+1)$-flat. Decompose $\cQ_L$ into disjoint parts $\cQ_{L, Q^+}$ according to $Q^+$.  In the
quotient $Q^+/U$, which is an affine $a$-subspace, the set
$(Q^+\cap W)/U$ is an affine line.  The members $Q$ of the corresponding
part $\cQ_{L, Q^+}$ are precisely the affine hyperplanes in $Q^+$ which meet this line in one
point. There are $2^a$ such hyperplanes, so the partition has $2^{a^2-a}$ parts.

We show that every part $\cQ_{L, Q^+}$ contains at least two bad flats.  Fix $Q^+$ and
write $C_0 = x+U$, $C_1 = x+w+U$, so that
\[
  Q^+\cap W=C_0\cup C_1,
  \qquad C_1=C_0+w.
\]
If
$|Q^+\cap A|\ne 2s$, then each complementary pair
$Q,Q+w$ contains a bad member.  This gives at least
$2^{a-1}\ge2$ bad flats.

Suppose instead that $|Q^+\cap A|=2s$.  Separate the flats in $\cQ_{L, Q^+}$
into two subsets $\cH_0$ and $\cH_1$ according to whether their intersection with
$W$ is $C_0$ or $C_1$.  Each family has size $2^{a-1}$.  A point of
$Q^+\setminus(C_0\cup C_1)$ lies in the same number of members of
$\cH_0$ and $\cH_1$, whereas a point of $C_i$ lies in every member of
$\cH_i$ and no member of $\cH_{1-i}$.  Therefore
\[
\begin{split}
  \sum_{Q\in\cH_0}|Q\cap A|
  -\sum_{Q\in\cH_1}|Q\cap A|
  =2^{a-1}\bigl(|C_0\cap A|-|C_1\cap A|\bigr),
\end{split}
\]
which is nonzero by \eqref{eq:unequal-cosets}.  Thus, not every member
of the part is good.  Finally, when $|Q^+\cap A|=2s$, a flat is bad if
and only if its complementary translate is bad.  There are therefore
at least two bad members in this case as well.

Multiplying two bad flats per part by the number of parts gives
$2\cdot2^{a^2-a}=2^{a^2-a+1}$.
\end{proof}

It remains to count the nice subspaces in $\widehat W$. We count ordered bases $v_1,\ldots,v_a$ for which $v_1$ is the unique nonzero point of $T$ in
their span. There are at least $2^{d-k}-1$ choices for $v_1$ since $|T|\ge 2^{d-k}$. Suppose $v_1,\ldots,v_i$ have been chosen.
We need to choose $v_{i+1}$ outside $T+\Span(v_1,\ldots,v_i)$, which has size at most $|T|2^i$. Thus, we have at least $2^{d+1}-2^{d-k+i+1}$ choices. Every nice $a$-subspace is counted exactly $\prod_{i=1}^{a-1}(2^a-2^i)$ times. Hence the number of nice subspaces is at least
\begin{equation}\label{eq:nice-count}
  \frac{(2^{d-k}-1)
    \prod_{i=1}^{a-1}(2^{d+1}-2^{d-k+i+1})}
  {\prod_{i=1}^{a-1}(2^a-2^i)}.
\end{equation}

Note that the families $\cQ_L$ for nice subspaces $L$ are
mutually disjoint, because $\dir(Q\cap W)=L^\perp$ determines $L$ for any $Q\in\cQ_L$.
Combining \eqref{eq:nice-count} with \cref{claim:forced-bad}, the
number of bad $d$-flats is at least
\begin{equation}\label{eq:numerator}
      \frac{(2^{d-k}-1)
    \prod_{i=1}^{a-1}(2^{d+1}-2^{d-k+i+1})}
  {\prod_{i=1}^{a-1}(2^a-2^i)}
  \cdot2^{a^2-a+1}.
\end{equation}
On the other hand, since $n=d+a$,
\begin{equation}\label{eq:all-flats}
\begin{split}
  |\Aff(n,d)|
  &=2^a\qbinom{n}{a}=2^{a(d+1)}
    \frac{\prod_{i=0}^{a-1}(1-2^{i-n})}
         {\prod_{i=1}^{a}(1-2^{-i})}.
\end{split}
\end{equation}
Dividing \eqref{eq:numerator} by \eqref{eq:all-flats} and simplifying the expression gives
\[
  \rho\ge
  \frac{2^{-k}(1-2^{-(d-k)})(1-2^{-a})
  \prod_{i=1}^{a-1}(1-2^{i-k})}
  {\prod_{i=0}^{a-1}(1-2^{i-n})}.
\]
The denominator is at most $1$, so discarding it proves the proposition.
\end{proof}

\begin{proof}[Proof of \cref{thm:main}]
Let $a=\lfloor k/2\rfloor$.  For $k$ sufficiently large, we have
$2\le a\le k$.  Since $\lambda^*(n,d,s)$ is non-increasing in $n$, we have
\[
  \lambda^*(d,s)\le\lambda^*(d+a,d,s).
\]
Applying \cref{prop:finite} and using
\[
\begin{split}
  (1-2^{-a})\prod_{i=1}^{a-1}(1-2^{i-k})
  &=1-O\left(2^{-a}+\sum_{i=1}^{a-1}2^{i-k}\right)=1-O(2^{-k/2})
\end{split}
\]
gives
\[
  \lambda^*(d,s)
  \le1-\bigl(1-2^{-(d-k)}\bigr)2^{-k}
  +O(2^{-3k/2}).\qedhere
\]
\end{proof}

\section{The $s=1$ case}\label{sec:singleton-proof}
In this section, we prove \cref{thm:singleton} using the standard special case of the
Gowers--Cauchy--Schwarz inequality~\cite{Gowers2001}.

Let $(f_\omega)_\omega$ be a tuple of functions $f_\omega:\F_2^n\to\R$ indexed by
$\omega\in\F_2^d$, and define
\[
  \Lambda_d((f_\omega)_\omega)
  :=\E_{x,v_1,\ldots,v_d}
    \prod_{\omega\in\F_2^d}
    f_\omega\left(x+\sum_{i=1}^d\omega_i v_i\right),
\]
where $x,v_1,\ldots,v_d$ are independent and uniform in $\F_2^n$.
The following is a special case of the Gowers--Cauchy--Schwarz inequality.

\begin{lemma}\label{lem:gcs}
Let $f:\F_2^n\to\R_{\ge0}$.  If one input of $\Lambda_d$ is the
constant function $1$ and the remaining $2^d-1$ inputs are $f$, then
\[
  \Lambda_d(1,f,\ldots,f)
  \le \Lambda_d(f,\ldots,f)^{(2^d-1)/2^d}.
\]
\end{lemma}

\begin{proof}
For a nonnegative function $g$, write
\[
  \|g\|_{U^d}:=\Lambda_d(g,\ldots,g)^{1/2^d}.
\]
The Gowers--Cauchy--Schwarz inequality states
\[
  |\Lambda_d((g_\omega)_\omega)|
  \le\prod_{\omega\in\F_2^d}\|g_\omega\|_{U^d}.
\]
Apply it with one $g_\omega=1$ and all the other
$g_\omega=f$.  Since $\|1\|_{U^d}=1$, the right-hand side is
$\|f\|_{U^d}^{2^d-1}$, which is the claimed expression.
\end{proof}

We now prove \cref{thm:singleton}. 

\begin{proof}[Proof of \cref{thm:singleton}]
    
Fix $A\subseteq V:=\F_2^n$ and
sample $x,v_1,\ldots,v_d\in V$ independently and uniformly. Let
$p$ be the probability that exactly one of the $2^d$ labeled
points
\[
  x+\sum_{i\in S}v_i,\qquad S\subseteq[d],
\]
lies in $A$. Let $\delta_{n,d}$ be the probability that the directions
$v_1,\ldots,v_d$ fail to be linearly independent. Then
\begin{equation}\label{eq:degenerate-probability}
  \delta_{n,d}
  =1-\prod_{i=0}^{d-1}(1-2^{i-n})
  \le\sum_{i=0}^{d-1}2^{i-n}<2^{d-n},
\end{equation}
and hence $\delta_{n,d}\rightarrow 0$ as $n$ goes to infinity.

Conditional on linear independence, their span is uniform in
$\Gr(n,d)$ and the resulting affine flat is uniform in $\Aff(n,d)$.
Thus, $p\ge(1-\delta_{n,d})\lambda^*(n,d,1,A)$, and hence
\begin{equation}\label{eq:flat-vs-labeled}
  \lambda^*(n,d,1,A)\le\frac{p}{1-\delta_{n,d}}.
\end{equation}

Put $f=\one_{A^c}$. By symmetry among the $2^d$ labeled points,
\[
  \frac{p}{2^d}
  =\Lambda_d(\one_A,f,\ldots,f).
\]
Since $\one_A=1-f$, multilinearity and \cref{lem:gcs} give
\begin{align*}
  \frac{p}{2^d}
  &=\Lambda_d(1,f,\ldots,f)-\Lambda_d(f,\ldots,f)\le
  \Lambda_d(f,\ldots,f)^{(2^d-1)/2^d}
  -\Lambda_d(f,\ldots,f).
\end{align*}
Set
\[
  u:=\Lambda_d(f,\ldots,f)^{1/2^d}\in[0,1].
\]
Then
\[
  \frac{p}{2^d}\le u^{2^d-1}(1-u).
\]
The right-hand side is maximized at $u=1-2^{-d}$.  Therefore
\begin{equation}\label{eq:labeled-upper}
  p\le(1-2^{-d})^{2^d-1}.
\end{equation}
Combining \eqref{eq:flat-vs-labeled} and
\eqref{eq:labeled-upper}, taking the maximum over $A$, and then letting
$n\to\infty$ proves
\[
  \lambda^*(d,1)\le(1-2^{-d})^{2^d-1}.\qedhere
\]

\end{proof}

\printbibliography

\end{document}